\documentclass{article}
\usepackage{graphicx} 
\usepackage{subcaption} 
\usepackage{color}
\usepackage[colorlinks=true,allcolors=blue
]{hyperref}
\usepackage{amsmath, amssymb, amsthm}
\usepackage{mathrsfs}
\usepackage{mathtools}
\usepackage{fullpage}
\usepackage[noadjust,nocompress]{cite}
\usepackage{graphics}
\usepackage{xcolor}
\usepackage{tikz}

\usepackage{thmtools}
\usepackage[noabbrev,capitalize,nameinlink]{cleveref}

\newtheorem{theorem}{Theorem}
\newtheorem{proposition}[theorem]{Proposition}
\newtheorem{lemma}[theorem]{Lemma}
\newtheorem{claim}[theorem]{Claim}
\newtheorem{corollary}[theorem]{Corollary}

\newtheorem{ques}[theorem]{Question}

\newcommand*{\myproofname}{Proof}
\newenvironment{claimproof}[1][\myproofname]{\begin{proof}[#1]}{\end{proof}}

\newcommand*{\floorfrac}[2]{\mathopen{}\left\lfloor\frac{#1}{#2}\right\rfloor\mathclose{}}
\newcommand*{\bceil}[1]{\left\lceil #1\right\rceil}
\newcommand*{\bfloor}[1]{\left\lfloor #1\right\rfloor}

\title{Ramsey number of a cycle versus a graph of a given size} 
\author{Stijn Cambie\thanks{Department of Computer Science, KU Leuven Campus Kulak-Kortrijk, 8500 Kortrijk, Belgium. Supported by a postdoctoral fellowship by the Research Foundation Flanders (FWO) with grant number 1225224N. Email: \protect\href{mailto:stijn.cambie@hotmail.com}{\protect\nolinkurl{stijn.cambie@hotmail.com}}} \and Andrea Freschi
\thanks{HUN-REN, Alfr{\'e}d R{\'e}nyi Institute of Mathematics, Budapest, Hungary. Research partially supported by ERC Advanced Grants ``GeoScape", no. 882971 and ``ERMiD", no. 101054936. E-mail: \protect\href{freschi.andrea@renyi.hu}{\protect\nolinkurl{freschi.andrea@renyi.hu}}} \and Patryk Morawski
\thanks{Department of Mathematics, ETH Z\"urich, Switzerland. Research supported in part by SNSF grant 200021-228014.
Email:  \protect\href{patryk.morawski@math.ethz.ch}{\protect\nolinkurl{patryk.morawski@math.ethz.ch}}. } \and Kalina Petrova
\thanks{Institute of Science and Technology Austria (ISTA), Klosterneurburg 3400, Austria. Supported by the European Union’s Horizon 2020 research and innovation programme
under the Marie Sk{\l}odowska-Curie grant agreement No 101034413 \includegraphics[width=5.5mm, height=4mm]{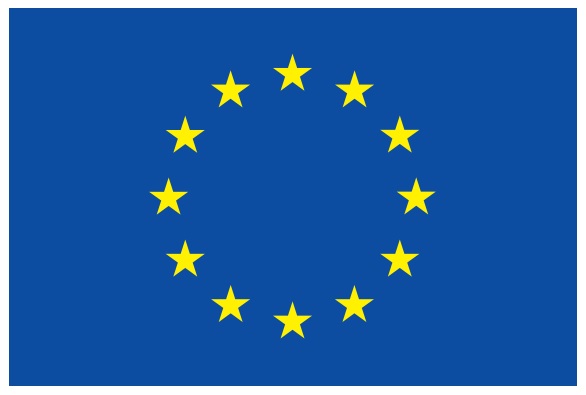}. Email: \protect\href{kalina.petrova@ist.ac.at}{\protect\nolinkurl{kalina.petrova@ist.ac.at}}. } \and Alexey Pokrovskiy
\thanks{Department of Mathematics, University College London, UK. Email: \protect\href{a.pokrovskiy@ucl.ac.uk}{\protect\nolinkurl{a.pokrovskiy@ucl.ac.uk}}}
}

\begin{document}

\maketitle

\begin{abstract}
    In this paper, we prove that for every $k$ and every graph $H$ with $m$ edges and no isolated vertices, the Ramsey number $R(C_k,H)$ is at most $2m+\floorfrac {k-1}2$, provided $m$ is sufficiently large with respect to $k$.  
    This settles a problem of Erd{\H{o}}s, Faudree, Rousseau and Schelp.
\end{abstract}

\section{Introduction}
For two graphs $G$ and $H$, the Ramsey number $R(G,H)$ is defined as the smallest $N$ such that in every red-blue colouring of the edges of~$K_N$ (i.e., the complete graph on~$N$ vertices), there is a red copy of $G$ or a blue copy of $H$.
The existence of these numbers was proven by Ramsey~\cite{ramsey1930}.
The first good quantitative bound $R(K_n,K_n)\le 4^n$ was later obtained by Erd\H{o}s and Szekeres~\cite{Erdös1935}. 
Obtaining better estimates for~$R(K_n,K_n)$ is a central  problem in combinatorics, and an upper bound of the form~$4^{(1-c)n}$ for some~$c>0$ has only been obtained recently~\cite{campos2023exponential}. 

Understanding Ramsey numbers of non-complete graphs is also important.
In particular, a substantial part of Ramsey theory research has focused on obtaining upper bounds on $R(G,G)$ in terms of various parameters of $G$. 
Note that the number of vertices of $G$, denoted as $|G|$, is not a very interesting parameter on its own: since $R(G,G)\le R(K_{|G|}, K_{|G|})$ for any graph $G$, proving bounds on $R(G,G)$ which only take~$|G|$ into account is equivalent to proving bounds on Ramsey numbers of complete graphs.
On the other hand, proving bounds on $R(G,G)$ in terms of the number of edges of $G$, denoted as $e(G)$, is a different and very interesting problem. 
The best known bound here is by Sudakov~\cite{sudakov2011conjecture} who proved that $R(G,G)\le 2^{250\sqrt{e\smash{(G)}}}$ for any graph $G$ with no isolated vertices, solving a well-known conjecture of Erd\H{o}s~\cite{erdos1984someproblems}. 

    Turning to the asymmetric case, researchers have been interested in obtaining bounds on $R(G,H)$ where $G$ is a fixed graph and $H$ varies. 
    In analogy to the previous paragraph, we again want to determine the best general upper bound in terms of $e(H)$ and no other parameters of $H$. 
    The simplest non-trivial case is when $G=K_3$.
    This problem was posed by Harary, who conjectured that $R(K_3, H)\le 2e(H)+1$ for any graph $H$ with no isolated vertices; this is tight when $H$ is a tree or a matching.
    Following weaker upper bounds by Erd\H{o}s, Faudree, Rosseau, and Schelp~\cite{erdos1987ramsey} and by Sidorenko~\cite{sidorenko1991upper}, Harary's Conjecture was fully proved via two independent proofs by  Goddard and Kleitman and by Sidorenko.
    \begin{theorem}[Goddard and Kleitman~\cite{goddard1994upper}; Sidorenko \cite{Sidorenko93}]\label{Theorem_triangle}
        For any graph $H$ with no isolated vertices, we have $R(K_3, H)\le 2e(H)+1$.
    \end{theorem}
    
    What about when we replace $K_3$ by some other graph $G$? 
    In~\cite{EFRS93}, Erd\H{o}s, Faudree, Rosseau, and Schelp initiated a systematic study of this question, proving many bounds and posing many open problems.
    They specifically introduced the notion of ``Ramsey size-linear" graphs, that is, graphs $G$ such that for some constant $C_G=C_G(G)$ we have $R(G,H)\le C_G\cdot e(H)$ for every graph $H$ with no isolated vertex.
    They discovered that the average degree of $G$ is largely responsible for it being Ramsey size-linear or not.
    In particular, they proved that graphs with $e(G)\ge 2|G|-2$ are not Ramsey size-linear, whereas graphs with $e(G)\le |G|+1$ always are.
    In the intermediate range, both behaviours are possible and determining which graphs are Ramsey size-linear is a difficult problem. The behaviour can be quite complicated --- for example Wigderson~\cite{Wigderson24+} showed that there are infinitely many graphs which are not Ramsey size-linear, but each of whose subgraphs is Ramsey size-linear, answering a question of Erd\H{o}s, Faudree, Rosseau, and Schelp. Even for specific simple $G$, like subdivisions of $K_4$, we still do not have a full answer~\cite{balister2002note,bradavc2024ramsey}. 

   For a given Ramsey size-linear graph $G$, one can go further and ask what the optimal constant $C_G$ should be.
   Generalizing Harary's original conjecture, Erd\H{o}s, Faudree, Rosseau, and Schelp posed the following question for the case where $G$ is a cycle $C_k$ on $k$ vertices:
   \begin{ques}[Erd\H{o}s, Faudree, Rosseau, and Schelp \cite{EFRS93}]\label{Conjecture_EFRS}
   For each $k$, is it true that all graphs $H$ with no isolated vertices and $e(H)$ sufficiently large with respect to~$k$ satisfy $R(C_k,H)\le 2e(H)+\floorfrac{k-1}{2} $?
    \end{ques}

    The upper bound in Question~\ref{Conjecture_EFRS} is tight, as seen by taking $H$ to be a matching (i.e., a collection of vertex-disjoint edges) and considering a red-blue colouring of the complete graph on $(2e(H)-1)+\floorfrac{k-1}{2}$ vertices consisting of a blue clique of order $2e(H)-1$ and all other edges red.
    Question~\ref{Conjecture_EFRS} appears as Problem $570$ in the database of Erd\H{o}s problems~\cite{BloomErdosProblems}, and was also previously listed by Chung and Graham in their book~\cite{chung1998erdos} and as problem $35$ in the database of hard graph theory Erd\H{o}s problems~\cite{ChungErdosProblems} from 2010.

    Erd\H{o}s, Faudree, Rosseau, and Schelp~\cite{EFRS93} verified Question~\ref{Conjecture_EFRS} when $k$ is even.
    When $k=3$, Question~\ref{Conjecture_EFRS} is just Harary's original conjecture which is true by Theorem~\ref{Theorem_triangle}.
    The case $k=5$ was resolved by Jayawardene~\cite{jayawardene1999thesis}. 
    In this paper, we settle all remaining cases.
    
    \begin{theorem}\label{theorem:main}
    For each odd $k\ge 7$, all graphs~$H$ with no isolated vertices satisfy $R(C_k, H)\le 2e(H)+\left \lfloor \frac{k-1}{2} \right \rfloor $, provided that $e(H)$ is sufficiently large with respect to~$k$.
    \end{theorem}

    Like the earlier proofs in this area by Sidorenko, Goddard and Kleitman, Erd\H{o}s, Faudree, Rosseau and Schelp, and Jayawardene, our proof is an induction. 
    We prove a more general upper bound (Theorem~\ref{thm:main+}) that holds for all values of $e(H)$ and equals $2e(H)+\floorfrac {k-1}2 $ for $e(H)$ sufficiently large. 
    The main proof is presented in~\cref{sec:mainproof}, and depends on multiple preliminary results.
    The latter are presented in~\cref{sec:prelim}. 
    Our core ideas depend on an estimate for $R(P_k,H)$ (Corollary~\ref{cor:estR(Pk,H)}), where~$P_k$ denotes the path on $k$ vertices, and on proving that if either the first or second neighbourhood of a vertex contains a copy of~$P_{2k}$ then one can find a copy of $C_k$ (Lemma~\ref{lemma:path_second_neighbourhood}).

    \smallskip
    
    \textbf{Notation and terminology.}
    For a graph $G$, we write $V(G)$ for the vertex set of $G$, $|G|$ for the number of vertices in $G$, and $e(G)$ for the number of edges in $G$.
    For a vertex $v \in V(G)$, we let $N_G(v)$ denote the neighbourhood of $v$ in $G$.
    For a set $S \subseteq V(G)$ we define the neighbourhood of $S$ to be the set $\{u \in V(G) \setminus S: \exists v \in S \text{ with }  uv \in E(G)\}$.
    The second neighbourhood of a vertex $v \in V(G)$ is the neighbourhood of $N(v)$ minus $v$ itself.
    If $G$ is red-blue coloured, we define the (second) red neighbourhood of $v \in V(G)$ to be the (second) neighbourhood in the subgraph of $G$ defined by the red edges of $G$.
    We similarly define the (second) blue neighbourhood.
    For any~$k\in\mathbb Z^+$, we let $[k]:=\{1,2,\dots,k\}$ denote the set of the first~$k$ positive integers.



\section{Preliminary results}\label{sec:prelim}


In this section, we list and prove some of the preliminary results which we need for the main proof in~\cref{sec:mainproof}.
The following two results can be found in \cite[Thm.~1]{EFRS78} and \cite[Thm.~1.3]{FL25} respectively, and will allow us to handle the cases where the graph~$H$ is either very dense or connected and sparse.

\begin{proposition}[\cite{EFRS78, Sudakov02}]
\label{prop:dense}
    For every $n\ge2$ and odd $k\ge3$, we have $R(C_k, K_n) \le (3k)\cdot n^{(k+1)/(k-1)}$.
\end{proposition}

\begin{lemma}[\cite{BEFRS82,FL25}]
\label{lemma:sparse_connected}
For every odd $k\ge3$, if $H$ is a connected graph on $n\ge(10k)^4$ vertices with average degree at most $2(1+(8k)^{-2})$ then $R(C_k,H)= 2n-1$.
\end{lemma}

As mentioned above, the main ingredient of our proof is a good estimate for $R(P_k,H)$, given by the following lemma.
\begin{lemma}
\label{lemma:ramsey_of_path_vs_any}
For all integers $k \geq 1$ and every graph $H$, we have $R(P_k, H) \leq |H| + k (\chi(H)-1)$.
\end{lemma}
\begin{proof}
    If $\chi(H)=1$, this is trivial. So we assume $\chi(H) \ge 2.$
    Next, we will prove by induction on $t$ that for every choice of $k,n_1,n_2, \ldots, n_t$ we have
    \begin{equation}\label{eq:RamseyPathMultipartite}
        R(P_k, K_{n_1, n_2, \ldots, n_t}) \le k(t-1)+ \sum_{i=1}^t n_i
    \end{equation}
    where $K_{n_1, n_2, \ldots, n_t}$ denotes the complete $t$-partite graph with parts of size $n_1,\dots,n_t$.
    Note that~\eqref{eq:RamseyPathMultipartite} implies the statement of the lemma, since~$H$ is a subgraph of $K_{n_1, n_2, \ldots, n_t}$ where $t= \chi(H)$ and the $n_i$ are the sizes of the colour classes in a $t$-colouring of $H$.
   H\"aggkvist~\cite{Haggkvist89} proved that $R(P_k, K_{n_1, n_2})\le k+n_1+n_2-2$, and so~\eqref{eq:RamseyPathMultipartite} holds if~$t=2$.
   Hence, we may assume that~$t\ge3$.
   Now, given a blue copy of~$K_{n_1+n_2+k, n_3, \ldots, n_t}$ within a red-blue coloured complete graph,
   we can apply H\"aggkvist's result to the graph spanned by the part of order~$n_1+n_2+k$ and find either a red copy of~$P_k$ or a blue copy of~$K_{n_1,n_2, n_3, \ldots, n_t}$.
   By this observation, and by the inductive hypothesis, we have
   $$R(P_k, K_{n_1, n_2, \ldots, n_t}) \le R(P_k, K_{n_1+n_2+k, n_3, \ldots, n_t})\le k(t-2)+(n_1+n_2+k)+\sum_{i=3}^t n_i,$$
   as desired.
\end{proof}

Combining Lemma~\ref{lemma:ramsey_of_path_vs_any} with the well-known fact that $e(H) \ge \binom{\chi(H)}{2} $, we obtain the following corollary.

\begin{corollary}\label{cor:estR(Pk,H)}
For all integers $k \geq 1$ and every graph $H$, we have $R(P_k, H) \leq |H| + k \sqrt{2e(H)}$.
\end{corollary}

We will want to use Corollary~\ref{cor:estR(Pk,H)} to find $H$, or a subgraph of $H$, within the first and second neighbourhoods of some suitable vertex in a graph $G$.
It is easy to see that if $G$ is $C_k$-free, then the subgraph of $G$ induced by the neighbourhood of any vertex is $P_k$-free.
The following lemma states that a similar statement 
is true for the second neighbourhood of any vertex.
 
\begin{lemma}
\label{lemma:path_second_neighbourhood}
Let $k\ge5$. 
For any graph $G$ and vertex $v\in V(G)$, if the second neighbourhood of $v$ contains a copy of $P_{2k}$ then $G$ contains a copy of $C_k$.
\end{lemma}


\begin{proof}
    Let the vertices of the copy of $P_{2k}$ be, in order, $v_1, v_2, \ldots, v_{2k}$.
    For every $i\in[2k]$, let $u_i$ be an arbitrary neighbour of $v$ adjacent to $v_i$.

    For every $j\in[k]$,  if $u_j\neq u_{j+k-4}$ then $vu_jv_jv_{j+1} \ldots v_{j+k-4}u_{j+k-4}v$ is a copy of $C_k$, see Figure~\ref{fig1:a}.
    Hence, we may assume that $u_1=u_{k-3}=u_{2k-7}$, $u_2=u_{k-2}=u_{2k-6}$ and $u_3=u_{k-1}=u_{2k-5}$.
    If $u_1\neq u_2$ then $u_1v_1v_2u_2v_{k-2}v_{k-1} \ldots v_{2k-7}u_1$ is a copy of $C_k$, see Figure~\ref{fig1:b}.
    Similarly, if $u_2\neq u_3$ then $u_2v_2v_3u_3v_{k-1}v_{k} \ldots v_{2k-6}u_2$ is a copy of $C_k$.
    If $u_1=u_3$ then $u_1=u_{k-1}$ and in particular $u_1v_1v_{2} \ldots v_{k-1}u_1$ is a copy of $C_k$, see Figure~\ref{fig1:c}.
    This concludes the proof of the lemma.
\end{proof}
\begin{figure}[h!]   
    \centering
        \begin{subfigure}{.32\textwidth}
        \centering
        \begin{tikzpicture}[scale=0.94]
        \draw (1.5,3.3) -- (4,0) -- (-1,0) -- (1.5,3.3);
        \filldraw (1.5,3.3) circle (2pt) node[anchor=south]{$v$};
        \filldraw[fill=white] (1.5,2) ellipse (1.5cm and 0.7cm);
        \filldraw[fill=white] (1.5,0) ellipse (2.5cm and 1cm);
        \filldraw[very thick] (0,0) circle (2pt) node[anchor=north]{$v_j$} -- (1,0) circle (2pt) node[anchor=north]{$v_{j+1}$} -- (1.5,0);
        \filldraw[very thick,dotted] (1.5,0) -- (2.5,0);
        \filldraw[very thick] (3,0) circle (2pt) node[anchor=north]{$v_{j+k-4}$} -- (2.5,0);
        \filldraw[very thick] (0,0) -- (1,2) circle (2pt) node[anchor=east]{$u_j$} -- (1.5,3.3) circle (2pt) node[anchor=south]{$v$} -- (2,2) circle (2pt) node[anchor=west]{$u_{j+k-4}$} -- (3,0); 
        \end{tikzpicture}
        \caption{}
        \label{fig1:a}
        \end{subfigure}   
        \hfill
        \begin{subfigure}{.32\textwidth}
        \centering
        \begin{tikzpicture}[scale=0.94]
        \draw (1.5,3.3) -- (4,0) -- (-1,0) -- (1.5,3.3);
        \filldraw (1.5,3.3) circle (2pt) node[anchor=south]{$v$};
        \filldraw[fill=white] (1.5,2) ellipse (1.5cm and 0.7cm);
        \filldraw[fill=white] (1.5,0) ellipse (2.5cm and 1cm);
        \filldraw[very thick] (1,2) circle (2pt) node[anchor=east]{$u_{1}$} -- (-0.2,0) circle (2pt) node[anchor=north]{$v_1$} -- (0.8,0) circle (2pt) node[anchor=north]{$v_{2}$} -- (2,2) circle (2pt) node[anchor=west]{$u_{2}$} -- (1.8,0) circle (2pt) node[anchor=north]{$v_{k-2}$} -- (2.2,0);
        \filldraw[very thick,dotted] (2.2,0) -- (2.9,0);
        \filldraw[very thick] (2.9,0) -- (3.3,0) circle (2pt) node[anchor=north]{$v_{2k-7}$} -- (1,2);  
        \end{tikzpicture}
        \caption{}
        \label{fig1:b}
        \end{subfigure}
        \hfill
        \begin{subfigure}{.32\textwidth}
        \centering
        \begin{tikzpicture}[scale=0.94]
        \draw (1.5,3.3) -- (4,0) -- (-1,0) -- (1.5,3.3);
        \filldraw (1.5,3.3) circle (2pt) node[anchor=south]{$v$};
        \filldraw[fill=white] (1.5,2) ellipse (1.5cm and 0.7cm);
        \filldraw[fill=white] (1.5,0) ellipse (2.5cm and 1cm);
        \filldraw[very thick] (0,0) circle (2pt) node[anchor=north]{$v_1$} -- (1,0) circle (2pt) node[anchor=north]{$v_{2}$} -- (1.5,0);
        \filldraw[very thick,dotted] (1.5,0) -- (2.5,0);
        \filldraw[very thick] (3,0) circle (2pt) node[anchor=north]{$v_{k-1}$} -- (2.5,0);
        \filldraw[very thick] (0,0) -- (1.5,2) circle (2pt) node[anchor=south]{$u_1=u_{k-1}$} -- (3,0);   
        \end{tikzpicture}
        \caption{}
        \label{fig1:c}
        \end{subfigure} 
        \vspace*{-0.25cm}
        \caption{From left to right, the cases $u_j\neq u_{j+k-4}$, $u_1\neq u_2$ and $u_1=u_{k-1}$ in the proof of~\cref{lemma:path_second_neighbourhood}.}
        \label{fig1}
    \end{figure}
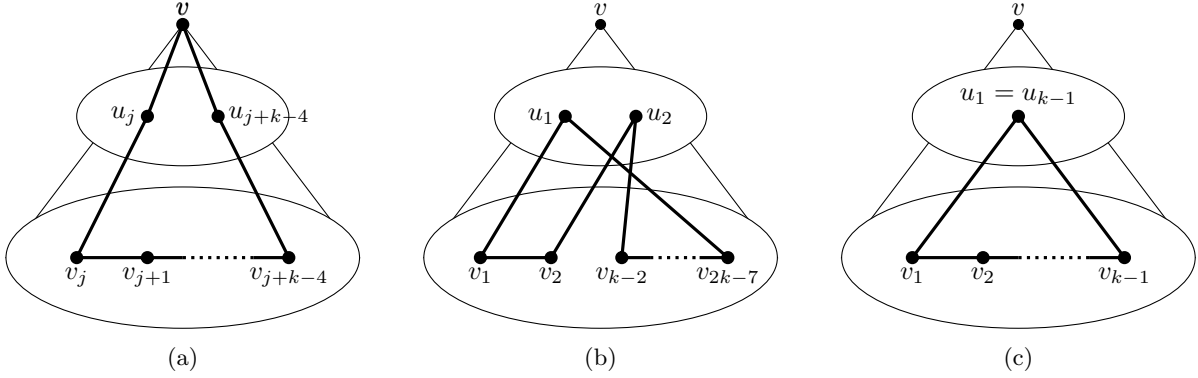

Finally, we need a separate argument for the case when $H$ is a matching.
We write~$mK_2$ for the matching consisting of~$m$ vertex-disjoint edges.

\begin{proposition}[Matching case]\label{proposition:matching}
    Let $m \ge k \geq 3$. 
    Then, $R(C_k,mK_2) = 2m + \floorfrac{k-1}{2} $.
\end{proposition}
\begin{proof}
    Let $N = 2m +\floorfrac{k-1}{2}$.
    For the lower bound, fix a subset $U \subseteq V(K_{N-1})$ of size $2m-1$ and consider the red-blue colouring of $K_{N-1}$ such that $\{u, v\} \in E(K_{N-1})$ is blue if and only if $\{u, v\} \subseteq U$.
    It is easy to check that this colouring contains neither a red $C_k$ nor a blue $mK_2$.

    For the upper bound, fix a colouring of $K_N$ and let $G$ be the blue graph.
    Assume that~$G$ contains no~$mK_2$.
    By the Tutte-Berge formula, there exists a vertex set $S\subseteq V(G)$ such that the size of the largest matching in $G$ is given by 
    \begin{equation}\label{eq:Tutte}
        \frac{1}{2}\left( N - odd(G - S) + |S| \right),    
    \end{equation} 
    where $odd(G-S)$ is the number of connected components of odd size in the graph $G - S$.
    Since~\eqref{eq:Tutte} is at most~$m-1$, we have $odd(G-S)\ge\floorfrac{k-1}{2} + 2+|S|$.
    This implies $|S|\le N/2$ since $odd(G-S)\le N-|S|$.
    Thus, $G-S$ has at least $N/2\ge k$ vertices and at least $\floorfrac{k+1}{2} + 1$ connected components.
    It follows that the complement of~$G$ contains a complete partite graph $F$ with~$k$ vertices and at least $\floorfrac{k+1}{2} + 1$ parts.
    Note that all edges in~$F$ are red and $F$ has minimum degree at least $\floorfrac{k+1}{2}$.
    By Dirac's theorem, $F$ is Hamiltonian.
    Therefore there is a red $C_k$ in the fixed colouring of $K_N$. 
\end{proof}

\section{Main proof}\label{sec:mainproof}

With the preliminary results in hand, we are ready to prove Theorem~\ref{theorem:main}.
As mentioned above, for induction it is more practical to prove a general bound, which holds for all values of~$e(H)$, and boils down to $2e(H)+\floorfrac {k-1}2 $ when $e(H)$ is sufficiently large.

\begin{theorem}\label{thm:main+}
    Let $k \geq 7$ be odd. 
    There exists a constant $B$ such that for any graph $H$ without isolated vertices we have that
    \[
        R(C_k, H) \leq 2e(H) + \max\left\{  \bfloor{B - \sqrt{e(H)}}, \bfloor{\frac{k}{2}} \right\}.
    \]
\end{theorem}

\begin{proof}
    Let $m_0$ be a large enough constant with respect to $k$ (e.g., $m_0 = 2^{63} k^{18}$ works) and set $B := (2m_0)^3$.
    We will prove the statement by induction on $e(H)$. 
    For the base case, where $e(H) \leq m_0$, we notice that the statement holds since
    \[
        R(C_k, H) \leq R(C_k, K_{2m_0}) \leq 3k \cdot (2m_0)^{(k+1)/(k-1)} \leq 3k \cdot (2m_0)^2 \leq B \leq 2e(H) + B - \sqrt{e(H)},
    \]
    where we used \cref{prop:dense} and that $m_0$ is large enough with respect to $k$.

    We now let $H$ be a graph on $n$ vertices with $m = e(H) > m_0$ edges and no isolated vertices.
    Suppose that the statement holds for every graph $H'$ with $e(H') < m$ and no isolated vertices.
    We let $N := 2m + \max\{ B - \bceil{\sqrt{m}}, \bfloor{\frac{k}{2}} \}$ and fix a red-blue colouring $G$ of $K_N$.
    Suppose for contradiction that $G$ contains neither a red copy of $C_k$ nor a blue copy of $H$.
    We first show that this cannot happen if $H$ is disconnected, too dense or too sparse.

    \begin{claim}\label{claim:connected+dense}
        $H$ is connected and 
        \[
            m^{2/3} \leq n \leq (1 - (20k)^{-2})m.
        \]
    \end{claim}
    \begin{claimproof}[Proof of Claim~\ref{claim:connected+dense}]
        Suppose first that $H$ is not connected.
        If $H$ is a matching, then~\cref{proposition:matching} implies~$G$ contains a red~$C_k$ or a blue~$H$, contradiction.
        Otherwise, $H$ contains a connected component $C$ with $2 \leq e(C) < m$ edges and $|C| \leq e(C) + 1$  vertices.
        By the inductive hypothesis, we have that
        \[
            R(C_k, C) \leq 2e(C) + \max \left\{ B - \sqrt{e(C)}, \bfloor{\frac{k}{2}} \right\} \leq N,
        \]
        where for the second inequality we used that $f(x) = 2x +B - \sqrt{x}$ is increasing for $x \geq 2$.
        Similarly,
        \begin{align*}
            R(C_k, H - C) &\leq 2(m -e(C)) + \max\left\{ \bfloor { B - \sqrt{m -e(C)} }, \bfloor{\frac{k}{2}} \right\}\\
            &\leq N - 2e(C)  + \max\left\{B - \bceil{\sqrt{m - e(C)}}, \bfloor{\frac{k}{2}}\right\} - \max\left\{ B - \bceil{\sqrt{m}}, \bfloor{\frac{k}{2}} \right\} \\
            &\leq N - e(C) -1 + \left( 1+ \bceil{\sqrt{m}} - \bceil{\sqrt{m - e(C)}} - e(C)  \right)\\
            &\leq N - |C|,
        \end{align*}
        where for the final inequality we used that
        \[
            \bceil{\sqrt{m}} - \bceil{\sqrt{m - e(C)}} \leq \bceil{\sqrt{m} - \sqrt{m - e(C)}}
        \]
        and that for $m \geq 3$ and $2 \leq x \leq m$ the function $f(x) = 1 + \bceil{\sqrt{m} - \sqrt{m-x}} -x$ takes its maximum at~$x=2$.
        Combining the two above observations, we can first find a blue copy $C'$ of $C$ in $G$ and then a blue copy of $H- C$ in $G - C'$ --- a contradiction to $G$ containing no blue copy of $H$.

        We can therefore assume that $H$ is connected.
        If $n \geq (1-(20k)^{-2})m$ 
        , since $m \geq m_0$ is large enough, we get $R(C_k, H) \leq 2n  -1 < N$ by \cref{lemma:sparse_connected}.
        On the other extreme, if $n \leq m^{2/3}$, then by \cref{prop:dense} we get
        \[
            R(C_k, H) \leq R(C_k, K_{m^{2/3}}) \leq 3k \cdot m^{2(k+1)/3(k-1)}\le   3k \cdot m^{8/9}\leq 2m < N,
        \]
        where we again used that $m \geq m_0$ is large enough.
        In either case, we get a contradiction --- which proves the claim.
    \end{claimproof}

    We now fix a vertex $v \in V(H)$ of minimum degree $\delta$.
    By the inductive hypothesis, we can find a blue copy of $H - v$ in $G$.
    Let $U = \{u_1, \dots, u_{\delta}\}$ be the images of $N_H(v)$  in this copy and let $S \subset V(G)$ be the vertices not in this copy.
    Note that $|S| = N - n + 1$ and that, since $G$ contains no blue copy of $H$, each vertex in $S$ must have at least one red neighbour in $U$.
    By the pigeonhole principle, there must exist a vertex $u \in U$ with at least $(N - n + 1) /\delta$ red neighbours in $G$.
    
    Let therefore $U_1$ be the red neighbourhood of $u$, let $\Pi$ be the second red neighbourhood of $u$ and let $U_2 = V(G) \setminus (\{u\} \cup U_1 \cup \Pi )$.
    Notice that by construction all edges between $U_1$ and $U_2$ in $G$ are blue.
    We will want to argue that we can embed some part of $H$ into $G[U_1]$ in blue and the rest into $G[U_2]$ --- which would give us a blue copy of $H$ in $G$, see Figure~\ref{fig:mainproof}.
    To that end, we first argue that both $U_1$ and $U_2$ are large enough.

    \begin{figure}
\centering
 \begin{tikzpicture}
        
        \fill[red!70] (0,0) -- (1.5,-1.5) -- (-1.5,-1.5) -- (0,0);
        \fill[red!30] (-1.5,-1.5) -- (-4.5,-3) -- (-2,-3)--(1.5,-1.5);
        \fill[blue!70] (1.5,-1.5) -- (4.5,-3) -- (2,-3)--(-1.5,-1.5);

        \filldraw[fill=white] (0,-1.5) ellipse (1.5cm and 0.75cm);
        \filldraw[fill=white] (0,-1.5) ellipse (1.2cm and 0.45cm);
        
        \filldraw (0,0) circle (2pt) node[anchor=south]{$u$};
        \filldraw[fill=white] (-3,-3) ellipse (1.5cm and 1cm);
        \filldraw[fill=white] (3,-3) ellipse (1.5cm and 1cm);

        \node at (-1.85,-0.75) {$U_1$};
        \node at (0,-1.5) {$H_1$};
        \node at (4.3,-2) {$U_2$};
        \node at (3,-3) {$\ge R(C_k,H_2)$};
        \node at (-3,-3) {$< R(P_{2k},H)$};
        \node at (-4.3,-2) {$\Pi$};
        \end{tikzpicture}
        \caption{
        In the proof of Theorem~\ref{thm:main+} we fix a vertex $u$ in our host graph with a large red neighbourhood $U_1$.
        By Lemma~\ref{lemma:ramsey_of_path_vs_any} we can embed a large part $H_1$ of $H$ into $U_1$.
        We can then show that either the second red neighbourhood $\Pi$ of $u$ has size at least $R(P_{2k}, H)$ --- in which case we can find a copy of $H$ in $\Pi$ --- or $U_2 = V(G) \setminus(U_1 \cup \Pi \cup \{u \})$ is large enough for us to find the rest of $H$ in there by induction.
        Note that by definition all the edges between $U_1$ and $U_2$ are blue.
    }

        \label{fig:mainproof}
\end{figure}
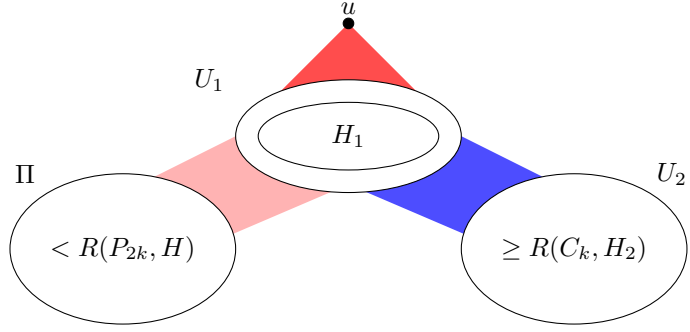

    \begin{claim}\label{claim:U_1bound}
        We have that
        \[
            \frac{n}{2} + k\sqrt{2m} \leq |U_1| \leq n + k\sqrt{2m}.
        \]
    \end{claim}
    \begin{claimproof}[Proof of Claim~\ref{claim:U_1bound}]
        For the lower bound let $d = 2m/n \geq 2 / (1 - (20k)^{-2})$ be the average degree in $H$ and notice that
        \[
            |U_1| \geq \frac{N - n + 1}{\delta} \geq \frac{2m  - n}{d} = n - \frac{n}{d} \geq \frac{n}{2} + k\sqrt{2m},
        \]
        where we used that $n \geq m^{2/3}$ is large enough with respect to $k$.

        For the upper bound, notice first that there is no red $P_{k-1}$ in $G[U_1]$, since any such path would form a red $C_k$ together with $u$ in $G$.
        Therefore, if $|U_1| \geq n + k\sqrt{2m}$, then by Corollary~\ref{cor:estR(Pk,H)} we can find a blue copy of $H$ in $G[U_1]$ --- a contradiction.
    \end{claimproof}

    We will want to embed $|U_1| - k\sqrt{2m}$ vertices of $H$ into $U_1$ and use the induction hypothesis to embed the rest of $H$ into $U_2$.
    The next claim says that  $U_2$ is large enough to do so.

    \begin{claim}\label{claim:U_2bound}
        We have that
        \[
            |U_2| \geq \max\left\{ n - |U_1| + k\sqrt{2m}, \hspace{5px}  2m \cdot \left( \frac{n - |U_1| + k\sqrt{2m}}{n} \right)^2 + B \right\}.
        \]
    \end{claim}
    \begin{claimproof}[Proof of Claim~\ref{claim:U_2bound}]
        We will want to use that $|U_2| = N -1 - |U_1| - |\Pi|$.
        To that end, we first notice that by \cref{lemma:path_second_neighbourhood} we get that $\Pi$ contains no red copy of $P_{2k}$.
        Therefore, we get that $|\Pi| \leq n  +2k\sqrt{2m}$ --- otherwise we could find a blue $H$ in $G[\Pi]$ by Corollary~\ref{cor:estR(Pk,H)}.

        Let now $\lambda = \frac{|U_1| - k\sqrt{2m}}{n}$ and notice that $1/2 \leq \lambda \leq 1$ by~\cref{claim:U_1bound}.
        We have that
        \begin{align*}
            |U_2| \geq 2m - 1 - |U_1| - n - 2k\sqrt{2m} \geq n - |U_1| + k\sqrt{2m} 
        \end{align*}
        and
        \begin{align*}
            |U_2| - 2m \cdot \left( \frac{n - |U_1|+ k\sqrt{2m}}{n} \right)^2 - B &\geq N - 2m \cdot \left( \frac{n - |U_1|+ k\sqrt{2m}}{n} \right)^2 - B - 1 - |U_1| - |\Pi|\\
            &\geq 2m \left( 1 - (1 - \lambda)^2 \right) - \bceil{\sqrt{m}} - 1 - |U_1| - n -2k\sqrt{2m}\\
            &\geq 2m \left(2\lambda - \lambda^2 - \frac{|U_1| - k\sqrt{2m}}{2m} \right) -n -4k\sqrt{2m}\\
            &\geq 2m \left(\frac{3}{2}\lambda - \lambda^2 \right) - n  - 4k\sqrt{2m}\\
            &\geq (20k)^{-2} m - 4k\sqrt{2m}\\
            &\geq 0,
        \end{align*}
        where we used that $n \leq (1 - (20k)^{-2})m$, that $m \geq m_0$ is large enough with respect to $k$ and that $\frac{3}{2}\lambda -\lambda^2 \geq 1/2$ for $1/2 \leq \lambda \leq 1$.
        This proves the claim.
    \end{claimproof}

    Finally, we fix a partition $V(H) = V_1 \cup V_2$ such that $|V_1| = |U_1| - k\sqrt{2m}$ and
    \[
        e(H[V_2]) \leq m \left( \frac{n -|V_1|}{n} \right)^2 = m \left( \frac{n - |U_1| + k\sqrt{2m}}{n} \right)^2.
    \]
    Such a partition exists, since the expected number of edges in a random set of order $n - |V_1|$ is at most the required bound.
    Since $G[U_1]$ contains no red $P_{k-1}$, by Corollary~\ref{cor:estR(Pk,H)} we can find a blue copy of $H[V_1]$ in $G[U_1]$.
    Let $H_2$ be the graph obtained from $H[V_2]$ by removing the isolated vertices.
    By the induction hypothesis, we also get that $ R(C_k, H_2) \leq 2e(H[V_2]) + B \leq |U_2|$.
    Since $|U_2| \geq n - |U_1| + k\sqrt{2m} = |V_2|$, we can therefore find a blue copy of $H[V_2]$ in $G[U_2]$.
    Finally, since all the edges between $U_1$ and $U_2$ in $G$ are blue, this gives us a blue copy of $H$ in $G$ and finishes the proof.

\end{proof}

\section*{Acknowledgement}
We thank the organizers of the online workshop \emph{Topics in Ramsey theory} of the Sparse Graphs Coalition\footnote{For more information, see \url{https://sparse-graphs.mimuw.edu.pl/doku.php}.}, where this project started.

\bibliographystyle{yuval}
\bibliography{ref}

\end{document}